\documentclass[a4paper,oneside,reqno]{amsart}
\usepackage{fancyhdr,mathrsfs,amssymb,amsmath,amsthm,latexsym}

\theoremstyle{plain}
\newtheorem*{theorem*}{Theorem}
\newtheorem{theorem}{Theorem}
\newtheorem{lemma}[theorem]{Lemma}

\newtheorem*{claim}{Claim}

\newtheorem{problem}[theorem]{Problem}

\theoremstyle{definition}

\theoremstyle{remark}

\usepackage{enumerate}
\usepackage[utf8]{inputenc}
\usepackage{pifont}
\usepackage{graphicx}
\usepackage{subcaption}
\usepackage{pdfpages}
\usepackage{wrapfig}
\usepackage{xcolor}
\usepackage[normalem]{ulem}
\usepackage{blindtext}
\usepackage{caption}
\usepackage{subcaption}
\usepackage{amsmath}
\usepackage{blindtext}
\usepackage{geometry}
\usepackage{fancyhdr}
\usepackage{tikz}
\usetikzlibrary{calc}
\usepackage{pgfplots}
\pgfplotsset{compat = newest}
\usepackage{mathtools}
\usepackage{paralist}
\usepackage{dirtytalk}

\title{Multipartite and Structural Results on Transparent Rectangle Visibility Graphs}

\author{Siraphob Buahong}
\address{Department of Mathematics, Faculty of Science, Chiang Mai University, Chiang Mai, Thailand}
\email{bsiraphob@hotmail.com}

\author{Teeradej Kittipassorn}
\address{Department of Mathematics and Computer Science, Faculty of Science, Chulalongkorn University, Bangkok 10330, Thailand, 
and Centre of Excellence in Mathematics, Ministry of Higher Education, Science, Research and Innovation, Thailand}
\email{teeradej.k@chula.ac.th}

\author{Jiratchaphat Nanta}
\address{Department of Mathematics, Faculty of Science, Chiang Mai University, Chiang Mai, Thailand}
\email{jiratchaphat\_nanta@cmu.ac.th}

\author{Piyashat Sripratak}
\address{Department of Mathematics, Faculty of Science, Chiang Mai University, Chiang Mai, Thailand}
\email{psripratak@gmail.com}

\author{Peerawit Suriya}
\address{Mathematics Institute, University of Warwick, Coventry, UK}
\email{peerawit.suriya@warwick.ac.uk}

\begin{document}
\pagestyle{plain}
\maketitle
\begin{abstract}
    We consider a graph representation in the plane, called the \textit{transparent rectangle visibility graph} (TRVG), where each vertex is represented by a rectangle in the plane with sides parallel to the plane axes, in a way that any two vertices are adjacent if and only if a vertical or horizontal line can be drawn from the interior of one rectangle to the other. Expanding upon previously done work by Juntarapomdach and Kittipassorn, we show that $K_{3,3,3}$ is not a TRVG, and classify complete $k$-partite TRVGs. We also prove that the complement of $C^2_n$ is not a TRVG whenever $n \geq 15$, and that every $k$-partite TRVG with $n$ vertices has at most $2(k-1)n-k(k-1)$ edges. Furthermore, we introduce a novel representation, the \textit{intersecting transparent rectangle visibility graph} (ITRVG), and show that there exists a graph that is an ITRVG but not a TRVG.
\end{abstract}

%%%%%%%%%%%%Intro%%%%%%%%%%%%%%
\section{Introduction}

The representation of graphs in various forms has been a long-standing topic of interest in the literature, especially ones that use rectilinear lines or shapes situated on a plane. 
The first representation of planar graphs via vertical and horizontal lines was introduced by Duchet, Hamidoune, Las Vergnas and Meyniel \cite{DHMH} in 1983, in which they used an S-representation.
An S-representation is a representation of a planar graph such that each vertex corresponds to a horizontal segment, and an edge between vertices $x$ and $y$ corresponds to a vertical segment which joins the $x$ and $y$ horizontal segments vertically, and does not intersect any other segments. 
The S-representation is later called a visibility representation in the literature \cite{TRTI, HZXH}.

A more generalized form of S-representation, the visibility graph \cite{LMW}, is a graph such that each vertex can be mapped to a vertical segment in the plane, in a way that for any two adjacent vertices in the graph, their corresponding segments `see' each other by an uninterrupted horizontal line of sight. 
Some authors switched the orientation of the segments and the visibility lines \cite{TA}. 
The main point is that the segments views each other via an orthogonal line.
This type of graph, and its generalizations, has been studied by various authors \cite{BEFH, AEGLST}, and has shown usage in the construction of very-large-scale integration (VLSI) systems \cite{SLMLW}.

Then, in 1976, Garey, Johnson, and So \cite{GJS} introduced the notion of rectangle visibility graphs (RVGs), in which they used unit squares instead of segments to represent vertices, and the same concept of visibility between two squares whenever their corresponding vertices are adjacent to each other. 
%Note that the unit squares' corners are situated on lattice points in the plane. 
Afterwards, there have been works on the properties of RVGs, such as the upper bound of edges in graphs with $n$ vertices \cite{HSV}, conditions for a bipartite graph $K_{m,n}$ to be an RVG \cite{DH}, classes of RVGs that can be decomposed into two caterpillar forests \cite{BDHS}, and more recently the area and size of the boundary of an RVG \cite{CDLN}. 

In this paper, instead of considering only an uninterrupted line of vision, we remove this restriction entirely and let rectangles see through one another. That is, we consider the rectangles to be transparent. 

We say that a graph is a \emph{transparent rectangle visibility graph} (TRVG) if it can be represented by a set of disjoint rectangles in the $xy$-plane with sides parallel to the axes, such that for any two vertices adjacent to one another, their corresponding rectangles see each other by a vertical or horizontal line of sight. Note that unlike RVGs, this line does not stop at its first intersection with another rectangle.

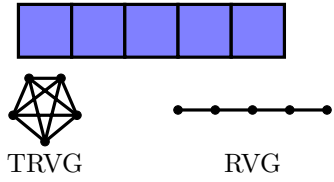
\begin{figure}[h]
    \centering
    \begin{tikzpicture}[scale=0.7, every node/.style={inner sep=1pt}]
        \filldraw[blue!50!white, draw=black, very thick] (0,0) rectangle (1,1);
        \filldraw[blue!50!white, draw=black, very thick] (1,0) rectangle (2,1);
        \filldraw[blue!50!white, draw=black, very thick] (2,0) rectangle (3,1);
        \filldraw[blue!50!white, draw=black, very thick] (3,0) rectangle (4,1);
        \filldraw[blue!50!white, draw=black, very thick] (4,0) rectangle (5,1);

        \coordinate (k1) at (0.2,-0.4);
        \coordinate (k2) at (0.8,-0.4);
        \coordinate (k3) at (1.1,-1.1);
        \coordinate (k4) at (0.5,-1.6);
        \coordinate (k5) at (-0.1,-1.1);

        \foreach \i in {1,...,5} {
            \filldraw (k\i) circle (0.08cm);
        }

        \draw[very thick] (k1)--(k2)--(k3)--(k4)--(k5)--(k1);
        \draw[very thick] (k1)--(k3) (k1)--(k4) (k2)--(k4)
                          (k2)--(k5) (k3)--(k5);

        \node at (0.5,-2.0) {TRVG};

        \coordinate (p1) at (3.0,-1.0);
        \coordinate (p2) at (3.7,-1.0);
        \coordinate (p3) at (4.4,-1.0);
        \coordinate (p4) at (5.1,-1.0);
        \coordinate (p5) at (5.8,-1.0);

        \foreach \i in {1,...,5} {
            \filldraw (p\i) circle (0.08cm);
        }

        \draw[very thick] (p1)--(p2)--(p3)--(p4)--(p5);

        \node at (4.4,-2.0) {RVG};
    \end{tikzpicture}
    \caption{A difference between a TRVG and an RVG}
    \label{fig:trvgandrvg}
\end{figure}

Some results on TRVGs have been obtained by Juntarapomdach and Kittipassorn ~\cite{CJTK}, that any threshold graph, cycle, tree are TRVGs, and also any rectangular, triangular, and hexagonal grid graphs. Furthermore, they have shown that bipartite TRVGs with $n$ vertices can have at most $2n-2$ edges, and that any $K_{p,q}$ is a TRVG if and only if $(p,q)$ has $\min\{p,q\} \leq 2$, or $(p,q) \in \{(3,3),(3,4)\}$.

We then further consider more graphs, such as the multipartite graphs, and the complement of the square of a cycle $C^2_n$. Furthermore, through allowing the intersections between rectangles, we define a new representation of graphs: \textit{intersecting rectangle visibility graphs} (ITRVG).

The rest of the paper is organized as follows. In Section~\ref{Complete_Bipartite}, we show that $K_{3,3,3}$ is a non-TRVG, and then we obtain the classification for a complete multipartite graph to be a TRVG. We prove that the complement of the square of a cycle $C_n^2$ is a non-TRVG for $n \ge 15$ in Section~\ref{complement}. In Section~\ref{k-partite}, we obtain an upper bound for the number of edges of a multipartite graph. We introduce a new class of graphs, called ITRVG, and show that TRVG is a proper subclass of this class in Section~\ref{irvg}. Finally, in Section~\ref{conclude}, we conclude the paper with some open problems.

\section{Complete Multipartite Graph}\label{Complete_Bipartite}

In the study of transparent rectangle visibility graphs (TRVGs) by Juntarapomdach and Kittipassorn \cite{CJTK}, one of the results is the TRVG classification of complete bipartite graphs. The theorem is as follows.

\begin{theorem}\label{thm:bipartite}
    For $p \leq q$, $K_{p,q}$ is a TRVG if and only if $p \leq 2$ or $(p,q) \in \{(3,3),(3,4)\}$.
\end{theorem}

In this section, we will explore further the complete $k$-partite graphs where $k \geq 3$ to determine which of them are TRVGs. The key result of this section is Lemma~\ref{lem:K_{3,3,3}} which is then used to classify other complete $k$-partite graphs in Theorem~\ref{thm:otherK}.

\begin{lemma}\label{lem:K_{3,3,3}}
    $K_{3,3,3}$ is a non-TRVG.
\end{lemma}
\begin{proof}
    Consider the first part of $K_{3,3,3}$. Up to translation, rotation, and reflection, there are two possible ways to construct blue rectangles, $B_1, B_2$ and $B_3$, representing these points (see Figure~\ref{fig:proofK333_1}). We then consider placements of points in the other parts through the concept of \textit{strips} which are the vertical areas and the horizontal areas of the gaps between each pair of consecutive blue rectangles (marked as $\alpha_1,\alpha_2,\beta_1,$ and $\beta_2$ in Figure~\ref{fig:proofK333_1}). A rectangle \textit{contains a vertical strip} if the left and right borders of the strip intersect the interior of the rectangle; therefore, the rectangle vertically sees both blue rectangles adjacent to the strip. The definition of containment is similar for horizontal strips. Note that if an other-color rectangle does not contain any strips, it can see at most one blue rectangle vertically and one blue rectangle horizontally. Thus, we have the properties of rectangle placement as follows: 
    \begin{compactitem}
        \item An other-color rectangle must contain at least one strip in order to see all three blue rectangles.
        \item The same-color rectangles cannot contain the same strip.
        \item If an other-color rectangle contains only one strip, it must see a blue rectangle that is not adjacent to this strip via another axis.
        \item One rectangle cannot see another rectangle in both axes.
    \end{compactitem} 

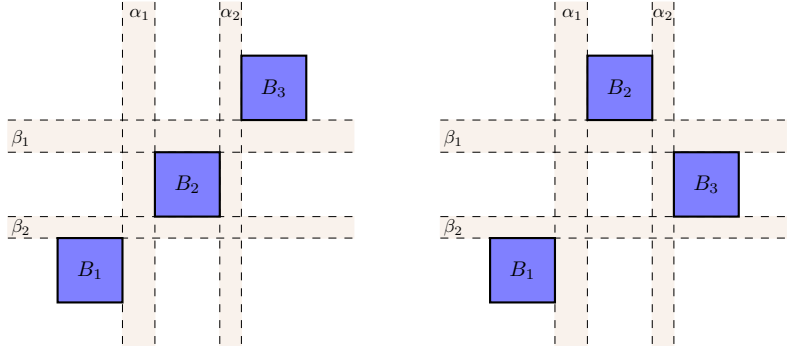
\begin{figure}[h]
    \centering
    \begin{tikzpicture}[scale = 0.2mm]
        \filldraw[brown!10!white] (-6.35,-4) rectangle (-5.6,4);
        \filldraw[brown!10!white] (-4.1,-4) rectangle (-3.6,4);
        \filldraw[brown!10!white] (-9,-1.5) rectangle (-1,-1);
        \filldraw[brown!10!white] (-9,0.5) rectangle (-1,1.25);
        
        \draw[dashed, very thin] (-6.35,4) -- (-6.35,-4);
        \draw[dashed, very thin] (-5.6,4) -- (-5.6,-4);

        \draw[dashed, very thin] (-4.1,4) -- (-4.1,-4);
        \draw[dashed, very thin] (-3.6,4) -- (-3.6,-4);

        \draw[dashed, very thin] (-9,-1.5) -- (-1,-1.5);
        \draw[dashed, very thin] (-9,-1) -- (-1,-1);

        \draw[dashed, very thin] (-9,0.5) -- (-1,0.5);
        \draw[dashed, very thin] (-9,1.25) -- (-1,1.25);

        \filldraw[blue!50!white, draw=black, thick] (-7.85,-3) rectangle (-6.35,-1.5);
        \draw (-7.1,-2.25) node[rectangle,scale=0.8] {$B_1$};
        \filldraw[blue!50!white, draw=black, thick] (-5.6,-1) rectangle (-4.1,0.5);
        \draw (-4.85,-0.25) node[rectangle,scale=0.8] {$B_2$};
        \filldraw[blue!50!white, draw=black, thick] (-3.6,1.25) rectangle (-2.1,2.75);
        \draw (-2.85,2) node[rectangle,scale=0.8] {$B_3$};
        \draw (-5.95,3.7) node[rectangle,scale=0.7] {$\alpha_1$};
        \draw (-3.85,3.7) node[rectangle,scale=0.7] {$\alpha_2$};
        \draw (-8.7,-1.25) node[rectangle,scale=0.7] {$\beta_2$};
        \draw (-8.7,0.825) node[rectangle,scale=0.7] {$\beta_1$};

        \filldraw[brown!10!white] (3.65,-4) rectangle (4.4,4);
        \filldraw[brown!10!white] (5.9,-4) rectangle (6.4,4);
        \filldraw[brown!10!white] (1,-1.5) rectangle (9,-1);
        \filldraw[brown!10!white] (1,0.5) rectangle (9,1.25);
        
        \draw[dashed, very thin] (3.65,4) -- (3.65,-4);
        \draw[dashed, very thin] (4.4,4) -- (4.4,-4);

        \draw[dashed, very thin] (5.9,4) -- (5.9,-4);
        \draw[dashed, very thin] (6.4,4) -- (6.4,-4);

        \draw[dashed, very thin] (1,-1.5) -- (9,-1.5);
        \draw[dashed, very thin] (1,-1) -- (9,-1);

        \draw[dashed, very thin] (1,0.5) -- (9,0.5);
        \draw[dashed, very thin] (1,1.25) -- (9,1.25);

        \filldraw[blue!50!white, draw=black, thick] (2.15,-3) rectangle (3.65,-1.5);
        \draw (2.9,-2.25) node[rectangle,scale=0.8] {$B_1$};
        \filldraw[blue!50!white, draw=black, thick] (4.4,1.25) rectangle (5.9,2.75);
        \draw (5.15,2) node[rectangle,scale=0.8] {$B_2$};
        \filldraw[blue!50!white, draw=black, thick] (6.4,-1) rectangle (7.9,0.5);
        \draw (7.15,-0.25) node[rectangle,scale=0.8] {$B_3$};
        \draw (4.05,3.7) node[rectangle,scale=0.7] {$\alpha_1$};
        \draw (6.15,3.7) node[rectangle,scale=0.7] {$\alpha_2$};
        \draw (1.3,-1.25) node[rectangle,scale=0.7] {$\beta_2$};
        \draw (1.3,0.825) node[rectangle,scale=0.7] {$\beta_1$};
    \end{tikzpicture}
    \caption{Possible rectangle representations of points in the first part}
    \label{fig:proofK333_1}
\end{figure}

\textbf{Case 1.} Consider the first configuration of placement of blue rectangles. In placing red rectangles representing points in the second part of $K_{3,3,3}$, by the pigeonhole principle, there are at least two red rectangles that see $B_2$ via the same axis. Without loss of generality, let them see $B_2$ vertically. Following placement properties above, one of these two red rectangles, say $R_1$, must contain only $\alpha_1$ and is above $B_2$ seeing $B_3$ horizontally, and the other, say $R_2$, contains only $\alpha_2$ and is below $B_2$ seeing $B_1$ horizontally (see Figure~\ref{fig:proofK333_2}). Using a similar argument with placement of yellow rectangles representing points in the third part, there are at least two yellow rectangles, namely $Y_1$ and $Y_2$, that see $B_2$ via the same axis.

\textbf{Case 1.1.} $Y_1$ and $Y_2$ see $B_2$ vertically. Following the placement properties, without loss of generality, we then have placements of $Y_1$ and $Y_2$ similarly to placements of $R_1$ and $R_2$, respectively (for example, see Figure~\ref{fig:proofK333_2}~(a)). As $Y_1$ is above $B_2$, its bottom side cannot be lower than the lower border of $\beta_1$ otherwise $Y_1$ will see $B_2$ horizontally, which is not possible. Similarly, $R_2$ which is set below $B_2$ must have its top side not higher than the upper border of $\beta_2$, and thus lower than the lower border of $\beta_1$. Using the lower border of $\beta_1$ as a dividing line, we have that $Y_1$ cannot see $R_2$ horizontally. Hence, in this case, $Y_1$ is forced to see $R_2$ vertically. For a similar reason, we also have that $Y_2$ must see $R_1$ vertically. This is impossible since doing so would lead to $Y_1$ seeing $Y_2$.

\textbf{Case 1.2.} $Y_1$ and $Y_2$ see $B_2$ horizontally. Again, following the placement properties, without loss of generality, $Y_1$ must contain only $\beta_1$ and is to the left of $B_2$ seeing $B_1$ vertically, and $Y_2$ must contain only $\beta_2$ and is to the right of $B_2$ seeing $B_3$ vertically (for example, see Figure~\ref{fig:proofK333_2}~(b)). As $Y_1$ is to the left of $B_2$, its right side cannot exceed the right border of $\alpha_1$. Note that $R_2$ cannot have its interior intersect with any vertical line of sight of $R_1$. This includes the right border of $\alpha_1$. Thus, together with $R_2$ being to the right of $R_1$, the left side of $R_2$ must be to the right of the right border of $\alpha_1$. Now, using the right border of $\alpha_1$ as a dividing line, we can say that $Y_1$ cannot see $R_2$ vertically. Hence, in this case, $Y_1$ must see $R_2$ horizontally. Similarly, we also find that $Y_2$ must see $R_1$ horizontally. Therefore, This case is also impossible since doing so would lead to $Y_1$ seeing $Y_2$.

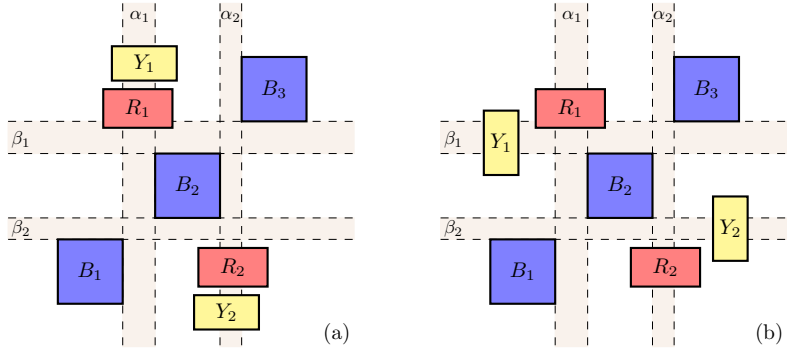
\begin{figure}[h]
    \centering
    \begin{tikzpicture}[scale = 0.2mm]
        \filldraw[brown!10!white] (-1.35,-4) rectangle (-0.6,4);
        \filldraw[brown!10!white] (0.9,-4) rectangle (1.4,4);
        \filldraw[brown!10!white] (-4,-1.5) rectangle (4,-1);
        \filldraw[brown!10!white] (-4,0.5) rectangle (4,1.25);
        
        \draw[dashed, very thin] (-1.35,4) -- (-1.35,-4);
        \draw[dashed, very thin] (-0.6,4) -- (-0.6,-4);

        \draw[dashed, very thin] (0.9,4) -- (0.9,-4);
        \draw[dashed, very thin] (1.4,4) -- (1.4,-4);

        \draw[dashed, very thin] (-4,-1.5) -- (4,-1.5);
        \draw[dashed, very thin] (-4,-1) -- (4,-1);

        \draw[dashed, very thin] (-4,0.5) -- (4,0.5);
        \draw[dashed, very thin] (-4,1.25) -- (4,1.25);

        \filldraw[blue!50!white, draw=black, thick] (-2.85,-3) rectangle (-1.35,-1.5);
        \draw (-2.1,-2.25) node[rectangle,scale=0.8] {$B_1$};
        \filldraw[blue!50!white, draw=black, thick] (-0.6,-1) rectangle (0.9,0.5);
        \draw (0.15,-0.25) node[rectangle,scale=0.8] {$B_2$};
        \filldraw[blue!50!white, draw=black, thick] (1.4,1.25) rectangle (2.9,2.75);
        \draw (2.15,2) node[rectangle,scale=0.8] {$B_3$};
        \draw (-0.95,3.7) node[rectangle,scale=0.7] {$\alpha_1$};
        \draw (1.15,3.7) node[rectangle,scale=0.7] {$\alpha_2$};
        \draw (-3.7,-1.25) node[rectangle,scale=0.7] {$\beta_2$};
        \draw (-3.7,0.825) node[rectangle,scale=0.7] {$\beta_1$};

        \filldraw[red!50!white, draw=black, thick] (-1.8,1.1) rectangle (-0.2,2);
        \draw (-1,1.55) node[rectangle,scale=0.8] {$R_1$};
        \filldraw[red!50!white, draw=black, thick] (0.4,-2.6) rectangle (2,-1.7);
        \draw (1.2,-2.15) node[rectangle,scale=0.8] {$R_2$};

        \filldraw[yellow!50!white, draw=black, thick] (-1.6,2.2) rectangle (-0.1,3);
        \draw (-0.85,2.6) node[rectangle,scale=0.8] {$Y_1$};
        \filldraw[yellow!50!white, draw=black, thick] (0.3,-3.6) rectangle (1.8,-2.8);
        \draw (1.05,-3.2) node[rectangle,scale=0.8] {$Y_2$};

        \filldraw[brown!10!white] (8.65,-4) rectangle (9.4,4);
        \filldraw[brown!10!white] (10.9,-4) rectangle (11.4,4);
        \filldraw[brown!10!white] (6,-1.5) rectangle (14,-1);
        \filldraw[brown!10!white] (6,0.5) rectangle (14,1.25);
        
        \draw[dashed, very thin] (8.65,4) -- (8.65,-4);
        \draw[dashed, very thin] (9.4,4) -- (9.4,-4);

        \draw[dashed, very thin] (10.9,4) -- (10.9,-4);
        \draw[dashed, very thin] (11.4,4) -- (11.4,-4);

        \draw[dashed, very thin] (6,-1.5) -- (14,-1.5);
        \draw[dashed, very thin] (6,-1) -- (14,-1);

        \draw[dashed, very thin] (6,0.5) -- (14,0.5);
        \draw[dashed, very thin] (6,1.25) -- (14,1.25);

        \filldraw[blue!50!white, draw=black, thick] (7.15,-3) rectangle (8.65,-1.5);
        \draw (7.9,-2.25) node[rectangle,scale=0.8] {$B_1$};
        \filldraw[blue!50!white, draw=black, thick] (9.4,-1) rectangle (10.9,0.5);
        \draw (10.15,-0.25) node[rectangle,scale=0.8] {$B_2$};
        \filldraw[blue!50!white, draw=black, thick] (11.4,1.25) rectangle (12.9,2.75);
        \draw (12.15,2) node[rectangle,scale=0.8] {$B_3$};
        \draw (9.05,3.7) node[rectangle,scale=0.7] {$\alpha_1$};
        \draw (11.15,3.7) node[rectangle,scale=0.7] {$\alpha_2$};
        \draw (6.3,-1.25) node[rectangle,scale=0.7] {$\beta_2$};
        \draw (6.3,0.825) node[rectangle,scale=0.7] {$\beta_1$};

        \filldraw[red!50!white, draw=black, thick] (8.2,1.1) rectangle (9.8,2);
        \draw (9,1.55) node[rectangle,scale=0.8] {$R_1$};
        \filldraw[red!50!white, draw=black, thick] (10.4,-2.6) rectangle (12,-1.7);
        \draw (11.2,-2.15) node[rectangle,scale=0.8] {$R_2$};

        \filldraw[yellow!50!white, draw=black, thick] (7,0) rectangle (7.8,1.5);
        \draw (7.4,0.75) node[rectangle,scale=0.8] {$Y_1$};
        \filldraw[yellow!50!white, draw=black, thick] (12.3,-2) rectangle (13.1,-0.5);
        \draw (12.7,-1.25) node[rectangle,scale=0.8] {$Y_2$};
        
        \draw (3.6,-3.7) node[rectangle,scale=0.8] {(a)};
        \draw (13.6,-3.7) node[rectangle,scale=0.8] {(b)};
    \end{tikzpicture}
    \caption{Examples of placement of red and yellow rectangles in  (a) Case 1.1 and (b) Case 1.2}
    \label{fig:proofK333_2}
\end{figure}
\begin{figure}[h]
    \centering
    \begin{tikzpicture}[scale = 0.2mm]
        \filldraw[brown!10!white] (-1.35,-4) rectangle (-0.6,4);
        \filldraw[brown!10!white] (0.9,-4) rectangle (1.4,4);
        \filldraw[brown!10!white] (-4,-1.5) rectangle (4,-1);
        \filldraw[brown!10!white] (-4,0.5) rectangle (4,1.25);
        
        \draw[dashed, very thin] (-1.35,4) -- (-1.35,-4);
        \draw[dashed, very thin] (-0.6,4) -- (-0.6,-4);

        \draw[dashed, very thin] (0.9,4) -- (0.9,-4);
        \draw[dashed, very thin] (1.4,4) -- (1.4,-4);

        \draw[dashed, very thin] (-4,-1.5) -- (4,-1.5);
        \draw[dashed, very thin] (-4,-1) -- (4,-1);

        \draw[dashed, very thin] (-4,0.5) -- (4,0.5);
        \draw[dashed, very thin] (-4,1.25) -- (4,1.25);

        \filldraw[blue!50!white, draw=black, thick] (-2.85,-3) rectangle (-1.35,-1.5);
        \draw (-2.1,-2.25) node[rectangle,scale=0.8] {$B_1$};
        \filldraw[blue!50!white, draw=black, thick] (-0.6,1.25) rectangle (0.9,2.75);
        \draw (0.15,2) node[rectangle,scale=0.8] {$B_2$};
        \filldraw[blue!50!white, draw=black, thick] (1.4,-1) rectangle (2.9,0.5);
        \draw (2.15,-0.25) node[rectangle,scale=0.8] {$B_3$};
        \draw (-0.95,3.7) node[rectangle,scale=0.7] {$\alpha_1$};
        \draw (1.15,3.7) node[rectangle,scale=0.7] {$\alpha_2$};
        \draw (-3.7,-1.25) node[rectangle,scale=0.7] {$\beta_2$};
        \draw (-3.7,0.825) node[rectangle,scale=0.7] {$\beta_1$};

        \filldraw[red!50!white, draw=black, thick] (-1.6,0.3) rectangle (-0.4,-0.7);
        \draw (-1,-0.2) node[rectangle,scale=0.8] {$R_1$};
        \filldraw[red!50!white, draw=black, thick] (0.7,-3) rectangle (2,-2.2);
        \draw (1.35,-2.6) node[rectangle,scale=0.8] {$R_2$};
        \filldraw[red!50!white, draw=black, thick] (-2.8,0.4) rectangle (-2,1.6);
        \draw (-2.4,1) node[rectangle,scale=0.8] {$R_3$};

         \filldraw[brown!10!white] (8.65,-4) rectangle (9.4,4);
        \filldraw[brown!10!white] (10.9,-4) rectangle (11.4,4);
        \filldraw[brown!10!white] (6,-1.5) rectangle (14,-1);
        \filldraw[brown!10!white] (6,0.5) rectangle (14,1.25);
        
        \draw[dashed, very thin] (8.65,4) -- (8.65,-4);
        \draw[dashed, very thin] (9.4,4) -- (9.4,-4);

        \draw[dashed, very thin] (10.9,4) -- (10.9,-4);
        \draw[dashed, very thin] (11.4,4) -- (11.4,-4);

        \draw[dashed, very thin] (6,-1.5) -- (14,-1.5);
        \draw[dashed, very thin] (6,-1) -- (14,-1);

        \draw[dashed, very thin] (6,0.5) -- (14,0.5);
        \draw[dashed, very thin] (6,1.25) -- (14,1.25);

        \filldraw[blue!50!white, draw=black, thick] (7.15,-3) rectangle (8.65,-1.5);
        \draw (7.9,-2.25) node[rectangle,scale=0.8] {$B_1$};
        \filldraw[blue!50!white, draw=black, thick] (9.4,1.25) rectangle (10.9,2.75);
        \draw (10.15,2) node[rectangle,scale=0.8] {$B_2$};
        \filldraw[blue!50!white, draw=black, thick] (11.4,-1) rectangle (12.9,0.5);
        \draw (12.15,-0.25) node[rectangle,scale=0.8] {$B_3$};
        \draw (9.05,3.7) node[rectangle,scale=0.7] {$\alpha_1$};
        \draw (11.15,3.7) node[rectangle,scale=0.7] {$\alpha_2$};
        \draw (6.3,-1.25) node[rectangle,scale=0.7] {$\beta_2$};
        \draw (6.3,0.825) node[rectangle,scale=0.7] {$\beta_1$};

        \filldraw[red!50!white, draw=black, thick] (8.4,0.3) rectangle (9.6,-0.7);
        \draw (9,-0.2) node[rectangle,scale=0.8] {$R_1$};
        \filldraw[red!50!white, draw=black, thick] (10.7,-3) rectangle (12,-2.2);
        \draw (11.35,-2.6) node[rectangle,scale=0.8] {$R_2$};
        \filldraw[red!50!white, draw=black, thick] (9.7,-2) rectangle (10.5,-0.8);
        \draw (10.1,-1.4) node[rectangle,scale=0.8] {$R_3$};
        \draw (3.6,-3.7) node[rectangle,scale=0.8] {(a)};
        \draw (13.6,-3.7) node[rectangle,scale=0.8] {(b)};
    \end{tikzpicture}
    \caption{Examples of placement of red rectangles in Case 2}
    \label{fig:proofK333_3}
\end{figure}
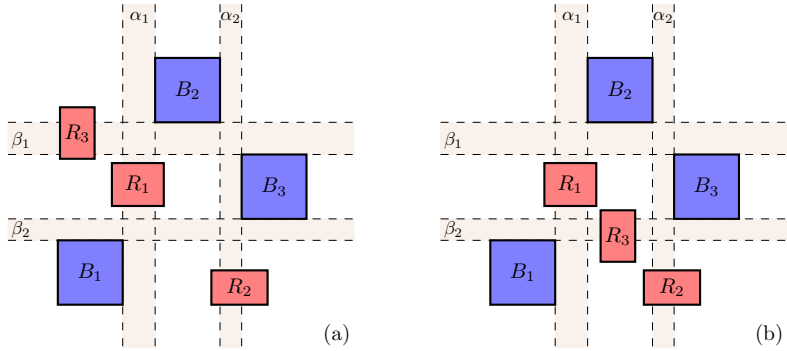
    \textbf{Case 2.}  Consider the second configuration of placement of blue rectangles. We will show that all red rectangles must be placed diagonally (similar to the first configuration of blue rectangles). Note that if there are two red rectangles that see $B_2$ horizontally, both of them must contain a horizontal strip adjacent to $B_2$. Since there is only one horizontal strip adjacent to $B_2$ and the same-color rectangles cannot share the same strip, this is impossible. Thus, there are at least two red rectangles, say $R_1$ and $R_2$, which see $B_2$ vertically. Without loss of generality, let $R_1$ contain only $\alpha_1$ and be between $B_1$ and $B_2$ seeing $B_3$ horizontally, and $R_2$ contain only $\alpha_2$ and be below $B_3$ seeing $B_1$ horizontally. Since both two vertical strips are already contained in $R_1$ and $R_2$, another red rectangle, $R_3$, must contain one horizontal strip (cannot contain both since $R_3$ cannot see $R_1$). Hence, there are two possible placements for $R_3$. One is astride $\beta_1$ above $B_1$ (see Figure~\ref{fig:proofK333_3}~(a)). Another is astride $\beta_2$ between $B_1$ and $B_3$ (see Figure~\ref{fig:proofK333_3}~(b)). Both of these lead to diagonal placement of red rectangles. Using a similar reason as for Case 1 with red rectangles instead of blue ones, we can conclude that it is also impossible to construct a representation of $K_{3,3,3}$ for this case.
\end{proof}

With Theorem~\ref{thm:bipartite} and Lemma~\ref{lem:K_{3,3,3}}, we are ready to prove Theorem~\ref{thm:otherK}. 

\begin{theorem}\label{thm:otherK}
    For $k \geq 3$, $K_{a_1,a_2,\dots,a_k}$ where $a_1 \leq a_2 \leq \dots \leq a_k$ is a TRVG if and only if $a_{k-1} \leq 2$ or $(a_{k-2},a_{k-1},a_k) \in \{(1,3,3),(1,3,4),(2,3,3),(2,3,4)\}$.
\end{theorem}

\begin{proof}
    Observe that for $K_{a_r,...,a_k}$ which is a TRVG, we can construct a representation of $K_{1,a_r,...,a_k}$ and $K_{2,a_r,...,a_k}$ using the idea of \textit{bounding box}, which is a box that covers all rectangles in a representation of a TRVG. Given a representation of $K_{a_r,...,a_k}$, we can add a rectangle that covers a side of the bounding box of $K_{a_r,...,a_k}$ to obtain a representation of $K_{1,a_r,...,a_k}$ (see Figure~\ref{fig:proofotherK_1}~(a)). Similarly, we can add two rectangles that do not see each other and cover two sides of the bounding box of $K_{a_r,...,a_k}$ to obtain a representation of $K_{2,a_r,...,a_k}$ (see Figure~\ref{fig:proofotherK_1}~(b)). 
    
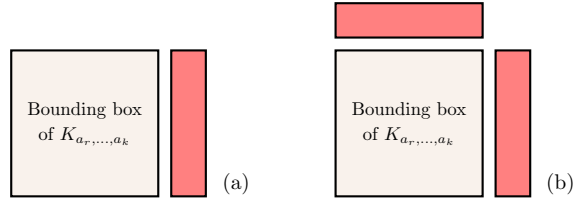
\begin{figure}[h]
    \centering
    \begin{tikzpicture}[scale = 0.2mm]
        \filldraw[brown!10!white, draw=black, thick] (-3.4,-3.4) rectangle (0,0);
        \draw (-1.7,-1.4) node[rectangle,scale=0.7] {Bounding box};
        \draw (-1.7,-2) node[rectangle,scale=0.7] {of $K_{a_r,...,a_k}$};
        \filldraw[red!50!white, draw=black, thick] (0.3,-3.4) rectangle (1.1,0);

        \filldraw[brown!10!white, draw=black, thick] (4.1,-3.4) rectangle (7.5,0);
        \draw (5.8,-1.4) node[rectangle,scale=0.7] {Bounding box};
        \draw (5.8,-2) node[rectangle,scale=0.7] {of $K_{a_r,...,a_k}$};
        \filldraw[red!50!white, draw=black, thick] (7.8,-3.4) rectangle (8.6,0);
        \filldraw[red!50!white, draw=black, thick] (4.1,0.3) rectangle (7.5,1.1);

        \draw (1.8,-3.1) node[rectangle,scale=0.8] {(a)};
        \draw (9.3,-3.1) node[rectangle,scale=0.8] {(b)};
    \end{tikzpicture}
    \caption{(a) A representation of $K_{1,a_r,...,a_k}$ and (b) a representation of $K_{2,a_r,...,a_k}$ constructed from a given representation of $K_{a_r,...,a_k}$}
    \label{fig:proofotherK_1}
\end{figure}

    If $a_{k-2} \geq 3$, we then have $3 \leq a_{k-1} \leq a_k$ which means that $K_{a_1,a_2,\dots,a_k}$ has $K_{3,3,3}$ as an induced subgraph. Note that an induced subgraph of a TRVG is also a TRVG. Hence, $K_{a_1,a_2,\dots,a_k}$ is a non-TRVG due to Lemma~\ref{lem:K_{3,3,3}}.

    On the other hand, if $a_{k-2} \leq 2$, then $a_1 \leq a_2 \leq \dots \leq a_{k-3} \leq 2$. Following the observation above, we can inductively construct a representation of $K_{a_1,a_2,\dots,a_k}$ from a representation of $K_{a_{k-1},a_k}$ if it is a TRVG (if it is not, we already know that $K_{a_1,a_2,\dots,a_k}$ is a non-TRVG since it contains $K_{a_{k-1},a_k}$ as an induced subgraph). Thus, $K_{a_1,a_2,\dots,a_k}$ is a TRVG if and only if $a_{k-2} \leq 2$ and $K_{a_{k-1},a_k}$ is a TRVG. According to Theorem~\ref{thm:bipartite}, the latter is equivalent to \say{$a_{k-1} \leq 2$ or $(a_{k-1},a_k) \in \{(3,3),(3,4)\}$}. Applying the fact that $a_{k-2} \leq a_{k-1}$, we therefore have $K_{a_1,a_2,\dots,a_k}$ is a TRVG if and only if $a_{k-1} \leq 2$ or $(a_{k-2},a_{k-1},a_k) \in \{(1,3,3),(1,3,4),(2,3,3),(2,3,4)\}$ as desired.
\end{proof}

%\begin{figure}[h]
%    \centering
%\begin{tikzpicture}[scale=0.6]
%    \draw[very thick] (-0.4,-0.4) rectangle (4.4,4.4);
%    \draw[step=0.5cm,gray,dashed] (-0.4,-0.4) grid (4.4,4.4);
%    \filldraw[blue!60!white, draw=black, very thick] (0,0) rectangle (3,3.5);
%    \draw (1.5,2.4) node[scale=0.8] {Bounding};
%    \draw (1.5,1.75) node[scale=0.8] {box of};
%    \draw (1.5,1.1) node[scale=0.8] {$G-v$};
%    \filldraw[red!60!white, draw=black,very thick] (3,0) rectangle (3.5,3.5);
%    \draw (3.25,1.75) node[scale=0.8] {$v$};
%    \filldraw[red!60!white, draw=black,very thick] (3.5,3.5) rectangle (4,4);
%    \draw (3.75,3.75) node[scale=0.8] {$v$};
 
%\end{tikzpicture}
%\caption{An addition of an isolated or universal vertex $v$}
%\label{fig:thresproof}
%\end{figure}

\section{Complement of $C_n^2$}\label{complement}
The 2\textsuperscript{nd} power of a cycle, $C_n^2$, is the graph with $V(G) = \{v_1, \ldots, v_n\}$ where any two distinct vertices $v_i$ and $v_j$ are adjacent if $|i-j| \leq 2$ or $n - |i-j| \leq 2$. Its complement is denoted by $D_n^2$. See Figure \ref{fig:d15square} for an example of $D_{15}^2$. The results of Juntarapomdach and Kittipassorn \cite{CJTK} show that $D_n^2$ is a TRVG for $n \leq 9$ and they conjectured that $D_n^2$ is a non-TRVG for $n \geq 10$.

\begin{figure}[h]
\centering
\begin{tikzpicture}[
    scale=1.2,
    every label/.style={font=\small, inner sep=3pt}
]
  \def\n{15}
  \def\r{2}
  \def\vr{2pt}

  \foreach \i in {1,...,\n} {
    \coordinate (v\i) at ({90 - 360/\n * (\i-1)}:\r);
  }

  \foreach \i in {1,...,\n} {
    \foreach \j in {1,...,\n} {
      \ifnum\i<\j
        \pgfmathtruncatemacro{\d}{abs(\i-\j)}
        \pgfmathtruncatemacro{\dist}{min(\d, \n-\d)}
        \ifnum\dist>2
          \draw (v\i) -- (v\j);
        \fi
      \fi
    }
  }

  \foreach \i in {1,...,\n} {
    \filldraw (v\i) circle[radius=\vr]
      node[label={90 - 360/\n * (\i-1)}:\(v_{\i}\)] {};
  }

\end{tikzpicture}
\caption{The graph \(D_{15}^2\).}
\label{fig:d15square}
\end{figure}
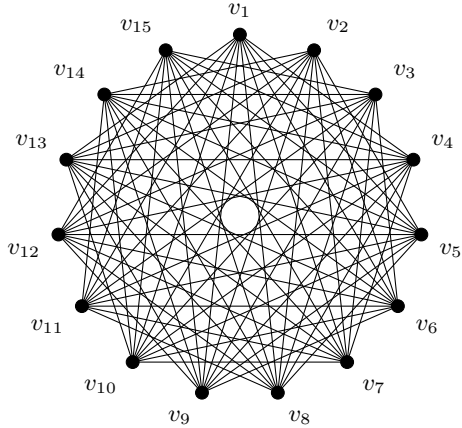

In this section, we prove Theorem \ref{thm:d2n} by using Lemma \ref{lem:K_{3,3,3}}.

\begin{theorem}\label{thm:d2n}
    $D^2_n$ is a non-TRVG when $n \geq 15$.
\end{theorem}
\begin{proof}
    Assume that $n \geq 15$. Let $v_i$ be a vertex of $D^2_n$ for $i \in \{1,\ldots,n\}$ and let $G$ be an induced subgraph of $D^2_n$ such that $V(G) = \{v_1,v_2,v_3,v_6,v_7,v_8,v_{11},v_{12},v_{13}\}$. We will show that $G$ is $K_{3,3,3}$. By the definition of $D^2_n$, there are no edges between any two vertices of $\{v_1,v_2,v_3\}$. The same holds for $\{v_6,v_7,v_8\}$ and $\{v_{11},v_{12},v_{13}\}$. Hence, sets $\{v_1,v_2,v_3\}$, $\{v_6,v_7,v_8\}$ and $\{v_{11},v_{12},v_{13}\}$ are independent. 
    
    Next, let $j\in \{1,2,3\}$, $k\in \{6,7,8\}$ and $l\in \{11,12,13\}$. Since $k-j \geq 3$, $l-k \geq 3$ and $n+j-l \geq 15 + j - l \geq 3$, there are edges between any two vertices from different independent sets. Therefore, $G$ is $K_{3,3,3}$. By Lemma \ref{lem:K_{3,3,3}}, we conclude that $D^2_n$ is a non-TRVG since $K_{3,3,3}$, which is non-TRVG, is an induced subgraph.
\end{proof}

\section{$k$-partite graph}\label{k-partite}
In the study by Juntarapomdach and Kittipassorn \cite{CJTK}, they obtained an upper bound for the number of edges of a bipartite graph with $n$ vertices which is $2n-2$ edges.

In this section, we generalize this upper bound to a $k$-partite graph with $n$ vertices as stated in Theorem \ref{thm:bound}.

To distinguish between each vertex of different parts, we assume that their rectangle representations must have different colors. In other words, we assume that vertices in the \(i^{\text{th}}\) part are represented by rectangles of color \(i\). We recall the following lemma from \cite{CJTK}, which will be useful for the proof of Theorem \ref{thm:bound}.

\begin{lemma}\label{lem:bound}
    Let $G$ be a bipartite TRVG, where $\mathcal{G}$ and $\mathcal{R}$ are the partite sets of $G$, represented by p green rectangles $g_1,g_2,\ldots,g_p$ in $\mathcal{G}$ and q red rectangles $r_1,r_2,\ldots,r_q$ in $\mathcal{R}$. If $g_i$ sees $\alpha_i$ red rectangles horizontally for $i \in \{1,2,\ldots,p\}$, then
    \[
    q \geq \alpha_1 + \alpha_2+ \ldots + \alpha_p - (p-1).
    \]
\end{lemma}

This lemma can be applied to bound the number of vertices in the $i^{\text{th}}$ part, which depends on the number of rectangles of color $i$ that are seen by each rectangle of color $j$ and the number of vertices in the $j^{\text{th}}$ part.

\begin{theorem}\label{thm:bound}
    Every $k$-partite TRVG with $n$ vertices has at most $2(k-1)n-k(k-1)$ edges.
\end{theorem}

\begin{proof}
   Let $n_1, n_2, \ldots, n_k$ be the numbers of vertices in each part and for $l \in \{1, 2, \ldots, n_i\}$, let $x^{(i,j)}_l$ and $y^{(i,j)}_l$ be the numbers of rectangles in the $j^{\text{th}}$ part that can be seen by rectangle $l$ in the $i^{\text{th}}$ part, horizontally and vertically respectively.
   
   We will apply Lemma \ref{lem:bound} to every pair of parts, both horizontally and vertically. Firstly, we consider when $i = 1$ and $j>i$. By Lemma~\ref{lem:bound}, we have that
   \begin{align*}
n_2 &\geq x^{(1,2)}_1 + x^{(1,2)}_2 + \cdots + x^{(1,2)}_{n_1} - (n_1 - 1), \\
n_2 &\geq y^{(1,2)}_1 + y^{(1,2)}_2 + \cdots + y^{(1,2)}_{n_1} - (n_1 - 1), \\
&\vdots \nonumber\\
n_k &\geq x^{(1,k)}_1 + x^{(1,k)}_2 + \cdots + x^{(1,k)}_{n_1} - (n_1 - 1), \\
n_k &\geq y^{(1,k)}_1 + y^{(1,k)}_2 + \cdots + y^{(1,k)}_{n_1} - (n_1 - 1).
\end{align*}

We further consider the case when $i \geq 2$ and $j > i$. Again, by Lemma~\ref{lem:bound}, we obtain that
\begin{align*}
n_3 &\geq x^{(2,3)}_1 + x^{(2,3)}_2 + \cdots + x^{(2,3)}_{n_2} - (n_2 - 1), \\
n_3 &\geq y^{(2,3)}_1 + y^{(2,3)}_2 + \cdots + y^{(2,3)}_{n_2} - (n_2 - 1), \\
&\vdots \nonumber\\
n_k &\geq x^{(k-1,k)}_1 + x^{(k-1,k)}_2 + \cdots + x^{(k-1,k)}_{n_{k-1}} - (n_{k-1} - 1), \\
n_k &\geq y^{(k-1,k)}_1 + y^{(k-1,k)}_2 + \cdots + y^{(k-1,k)}_{n_{k-1}} - (n_{k-1} - 1).
\end{align*}

We then combine all equations together and obtain that
\[
2\sum_{i = 2}^k (i-1)n_i \geq\sum_{i=1}^{k-1} \sum_{j=i+1}^{k} \sum_{l=1}^{n_i} \left( x^{(i,j)}_l + y^{(i,j)}_l \right) -  2 \sum_{i=1}^ k (k-i)n_i +k(k-1).
\]
We can see that 
\[
\sum_{i=1}^{k-1} \sum_{j=i+1}^{k} \sum_{l=1}^{n_i} \left( x^{(i,j)}_l + y^{(i,j)}_l \right)
\] 
is equal to the number of edges in this graph, since it is the total number of pairs of rectangles seeing each other, and each pair contributes exactly one edge. Hence, we obtain that
\[
2\sum_{i = 2}^k (i-1)n_i \geq e(G) -  2 \sum_{i=1}^ k (k-i)n_i +k(k-1).
\]
By rearranging, we obtain that
%\[
%2\sum_{i = 1}^k (k-1)n_i - k(k-1) \geq e(G).
%\]
%Therefore, we conclude that
\[
2 (k-1)n - k(k-1) \geq e(G). \qedhere 
\]
\end{proof}
\section{Intersecting Rectangle Visibility Graph}\label{irvg}
As we explore on the subject of Intersect Transparent Rectangle Visibility Graph (ITRVG), we see that the idea of ITRVG construction is relatively close to TRVG. On the one hand, from the definition, it follows immediately that any TRVG representation of a graph is also an ITRVG representation. Thus, any TRVG is an ITRVG. On the other hand, when we have an ITRVG, we do not know whether it is a TRVG or not. Initially, they seemed to be TRVGs as well since we have found only ITRVGs that can be transformed to a TRVG representation of that graph. However, because of the advantages of being intersected, we can find an example of an ITRVG that is proven to be a non-TRVG using the similar idea from Lemma \ref{lem:K_{3,3,3}} as follows. 

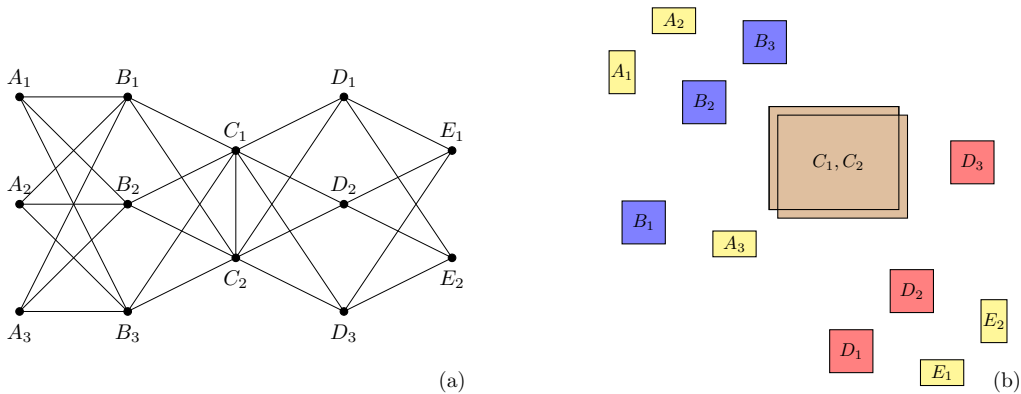
\begin{figure}[h]
    \centering
\begin{tikzpicture}[scale = 0.25mm]
    \filldraw[black] (-4,2) circle (2pt);
    \draw (-4,2.35) node[rectangle,scale=0.8] {$A_1$};    
    \filldraw[black] (-4,0) circle (2pt);
    \draw (-4,0.35) node[rectangle,scale=0.8] {$A_2$};
    \filldraw[black] (-4,-2) circle (2pt);
    \draw (-4,-2.4) node[rectangle,scale=0.8] {$A_3$};
    
    \draw[very thin] (-4,2) -- (-2,2);
    \draw[very thin] (-4,2) -- (-2,0);
    \draw[very thin] (-4,2) -- (-2,-2);
    \draw[very thin] (-4,0) -- (-2,2);
    \draw[very thin] (-4,0) -- (-2,0);
    \draw[very thin] (-4,0) -- (-2,-2);
    \draw[very thin] (-4,-2) -- (-2,2);
    \draw[very thin] (-4,-2) -- (-2,0);
    \draw[very thin] (-4,-2) -- (-2,-2);
    
    \filldraw[black] (-2,2) circle (2pt);
    \draw (-2,2.35) node[rectangle,scale=0.8] {$B_1$};
    \filldraw[black] (-2,0) circle (2pt);
    \draw (-2,0.35) node[rectangle,scale=0.8] {$B_2$};
    \filldraw[black] (-2,-2) circle (2pt);
    \draw (-2,-2.4) node[rectangle,scale=0.8] {$B_3$};

    \draw[very thin] (-2,2) -- (0,1);
    \draw[very thin] (-2,2) -- (0,-1);
    \draw[very thin] (-2,0) -- (0,1);
    \draw[very thin] (-2,0) -- (0,-1);
    \draw[very thin] (-2,-2) -- (0,1);
    \draw[very thin] (-2,-2) -- (0,-1);
    
    \filldraw[black] (0,1) circle (2pt);
    \draw (0,1.35) node[rectangle,scale=0.8] {$C_1$};
    \draw[very thin] (0,-1) -- (0,1);
    \filldraw[black] (0,-1) circle (2pt);
    \draw (0,-1.4) node[rectangle,scale=0.8] {$C_2$};

    \draw[very thin] (2,2) -- (0,1);
    \draw[very thin] (2,2) -- (0,-1);
    \draw[very thin] (2,0) -- (0,1);
    \draw[very thin] (2,0) -- (0,-1);
    \draw[very thin] (2,-2) -- (0,1);
    \draw[very thin] (2,-2) -- (0,-1);
    
    \filldraw[black] (2,2) circle (2pt);
    \draw (2,2.35) node[rectangle,scale=0.8] {$D_1$};
    \filldraw[black] (2,0) circle (2pt);
    \draw (2,0.35) node[rectangle,scale=0.8] {$D_2$};
    \filldraw[black] (2,-2) circle (2pt);
    \draw (2,-2.4) node[rectangle,scale=0.8] {$D_3$};

    \draw[very thin] (2,2) -- (4,1);
    \draw[very thin] (2,2) -- (4,-1);
    \draw[very thin] (2,0) -- (4,1);
    \draw[very thin] (2,0) -- (4,-1);
    \draw[very thin] (2,-2) -- (4,1);
    \draw[very thin] (2,-2) -- (4,-1);
    
    \filldraw[black] (4,1) circle (2pt);
    \draw (4,1.35) node[rectangle,scale=0.8] {$E_1$};
    \filldraw[black] (4,-1) circle (2pt);
    \draw (4,-1.4) node[rectangle,scale=0.8] {$E_2$};

    \draw[white] (-1,-2.5) rectangle (1,-3.3);
    \draw (4,-3.3) node[rectangle,scale=0.8] {(a)};
\end{tikzpicture}
\hspace{1.5cm}
\begin{tikzpicture}[scale = 0.2mm]
    \filldraw[brown!50!white, draw=black, thin] (-1.5,0.8) rectangle (1.5,-1.6);
    \filldraw[brown!50!white, draw=black, thin] (-1.3,0.6) rectangle (1.7,-1.8);
    \draw (0.1,-0.5) node[rectangle,scale=0.7] {$C_1,C_2$};
    \draw[thin] (-1.5,0.8) rectangle (1.5,-1.6);

    \filldraw[blue!50!white, draw=black, thin] (-2.1,1.8) rectangle (-1.1,2.8);
    \draw (-1.6,2.3) node[rectangle,scale=0.7] {$B_3$};
    \filldraw[blue!50!white, draw=black, thin] (-3.5,0.4) rectangle (-2.5,1.4);
    \draw (-3,0.9) node[rectangle,scale=0.7] {$B_2$};
    \filldraw[blue!50!white, draw=black, thin] (-4.9,-2.4) rectangle (-3.9,-1.4);
    \draw (-4.4,-1.9) node[rectangle,scale=0.7] {$B_1$};

    \filldraw[yellow!50!white, draw=black, thin] (-2.8,-2.7) rectangle (-1.8,-2.1);
    \draw (-2.3,-2.4) node[rectangle,scale=0.7] {$A_3$};
    \filldraw[yellow!50!white, draw=black, thin] (-5.2,1.1) rectangle (-4.6,2.1);
    \draw (-4.9,1.6) node[rectangle,scale=0.7] {$A_1$};
    \filldraw[yellow!50!white, draw=black, thin] (-4.2,2.5) rectangle (-3.2,3.1);
    \draw (-3.7,2.8) node[rectangle,scale=0.7] {$A_2$};

    \filldraw[red!50!white, draw=black, thin] (2.7,-1) rectangle (3.7,0);
    \draw (3.2,-0.5) node[rectangle,scale=0.7] {$D_3$};
    \filldraw[red!50!white, draw=black, thin] (1.3,-4) rectangle (2.3,-3);
    \draw (1.8,-3.5) node[rectangle,scale=0.7] {$D_2$};
    \filldraw[red!50!white, draw=black, thin] (-0.1,-5.4) rectangle (0.9,-4.4);
    \draw (0.4,-4.9) node[rectangle,scale=0.7] {$D_1$};

    \filldraw[yellow!50!white, draw=black, thin] (2,-5.7) rectangle (3,-5.1);
    \draw (2.5,-5.4) node[rectangle,scale=0.7] {$E_1$};
    \filldraw[yellow!50!white, draw=black, thin] (3.4,-4.7) rectangle (4,-3.7);
    \draw (3.7,-4.2) node[rectangle,scale=0.7] {$E_2$};
    \draw (4,-5.6) node[rectangle,scale=0.8] {(b)};
\end{tikzpicture}
    \caption{(a) Graph $G$ and (b) an ITRVG representation of $G$}
    \label{fig:ITRVGexam}
\end{figure}

\begin{lemma}
    The graph G in Figure \ref{fig:ITRVGexam} is an ITRVG; however, it is a non-TRVG.
\end{lemma}

\begin{proof}
    In this proof, in addition to the usual naming of each rectangle, we will denote the type of rectangles by the letter of the nodes they represent; for example, the rectangles representing the nodes $A_1$, $A_2$, and $A_3$ will be called $A$-rectangles. As graph $G$ has an ITRVG representation shown in Figure \ref{fig:ITRVGexam} (b), it is an ITRVG. To show that $G$ is a non-TRVG, let $G'$ be the induced subgraph of $G$ such that $V(G') = \{A_1,A_2,A_3,B_1,B_2,B_3,C_1,C_2\}$. As discussed in the proof of Lemma \ref{lem:K_{3,3,3}}, there are two possible ways to place all $B$-rectangles, up to translation, rotation, and reflection (see Figure \ref{fig:proofK333_1}).

    We then consider the possible placement of rectangles representing the other nodes of $G'$ using the idea and properties of strips in the proof of Lemma \ref{lem:K_{3,3,3}}. Note that there are four strips as well as four sets of rectangles that every rectangle in each set must not see any rectangles in the other sets, namely $\{A_1\}$, $\{A_2\}$, $\{A_3\}$, and $\{C_1,C_2\}$, while having to see all $B_1, B_2$, and $B_3$ rectangles. Thus, rectangles from each set must contain exactly one strip, different from the strips of the other sets. Without loss of generality, this leads to only three possible representations of $G'$ as shown in Figure \ref{fig:TRVG_of_G'}.

\begin{figure}[h]
    \centering
    \begin{tikzpicture}[scale = 0.195mm]
        \filldraw[brown!10!white] (-6.35,-4) rectangle (-5.6,4);
        \filldraw[brown!10!white] (-4.1,-4) rectangle (-3.6,4);
        \filldraw[brown!10!white] (-9,-1.5) rectangle (-1,-0.75);
        \filldraw[brown!10!white] (-9,0.75) rectangle (-1,1.25);
        
        \draw[dashed, very thin] (-6.35,4) -- (-6.35,-4);
        \draw[dashed, very thin] (-5.6,4) -- (-5.6,-4);

        \draw[dashed, very thin] (-4.1,4) -- (-4.1,-4);
        \draw[dashed, very thin] (-3.6,4) -- (-3.6,-4);

        \draw[dashed, very thin] (-9,-1.5) -- (-1,-1.5);
        \draw[dashed, very thin] (-9,-0.75) -- (-1,-0.75);

        \draw[dashed, very thin] (-9,0.75) -- (-1,0.75);
        \draw[dashed, very thin] (-9,1.25) -- (-1,1.25);

        \filldraw[blue!50!white, draw=black, thick] (-7.85,-3) rectangle (-6.35,-1.5);
        \draw (-7.1,-2.25) node[rectangle,scale=0.8] {$B_1$};
        \filldraw[blue!50!white, draw=black, thick] (-5.6,-0.75) rectangle (-4.1,0.75);
        \draw (-4.85,0) node[rectangle,scale=0.8] {$B_2$};
        \filldraw[blue!50!white, draw=black, thick] (-3.6,1.25) rectangle (-2.1,2.75);
        \draw (-2.85,2) node[rectangle,scale=0.8] {$B_3$};
        \draw (-5.95,3.7) node[rectangle,scale=0.7] {$\alpha_1$};
        \draw (-3.85,3.7) node[rectangle,scale=0.7] {$\alpha_2$};
        \draw (-8.7,-1.125) node[rectangle,scale=0.7] {$\beta_2$};
        \draw (-8.7,1) node[rectangle,scale=0.7] {$\beta_1$};

        \filldraw[yellow!50!white, draw=black, thick] (-6.6,2.2) rectangle (-5.1,3);
        \draw (-5.85,2.6) node[rectangle,scale=0.8] {$A_2$};
        \filldraw[yellow!50!white, draw=black, thick] (-4.95,-3.6) rectangle (-3.45,-2.8);
        \draw (-4.2,-3.2) node[rectangle,scale=0.8] {$A_3$};
        \filldraw[yellow!50!white, draw=black, thick] (-8,0.4) rectangle (-7.2,1.9);
        \draw (-7.6,1.15) node[rectangle,scale=0.8] {$A_1$};

        \filldraw[brown!50!white, draw=black, thick] (-3.3,-2.5) rectangle (-2.5,-0.5);
        \draw (-2.9,-1.5) node[rectangle,scale=0.8] {$C_1$};
        \filldraw[brown!50!white, draw=black, thick] (-2.3,-2) rectangle (-1.5,0.1);
        \draw (-1.9,-0.95) node[rectangle,scale=0.8] {$C_2$};
        \draw (-1.4,-3.7) node[rectangle,scale=0.8] {(a)};
    \end{tikzpicture}
    \hspace{0.2cm}
    \begin{tikzpicture}[scale = 0.195mm]
        \filldraw[brown!10!white] (3.65,-4) rectangle (4.4,4);
        \filldraw[brown!10!white] (5.9,-4) rectangle (6.4,4);
        \filldraw[brown!10!white] (1,-1.5) rectangle (9,-1);
        \filldraw[brown!10!white] (1,0.5) rectangle (9,1.25);
        
        \draw[dashed, very thin] (3.65,4) -- (3.65,-4);
        \draw[dashed, very thin] (4.4,4) -- (4.4,-4);

        \draw[dashed, very thin] (5.9,4) -- (5.9,-4);
        \draw[dashed, very thin] (6.4,4) -- (6.4,-4);

        \draw[dashed, very thin] (1,-1.5) -- (9,-1.5);
        \draw[dashed, very thin] (1,-1) -- (9,-1);

        \draw[dashed, very thin] (1,0.5) -- (9,0.5);
        \draw[dashed, very thin] (1,1.25) -- (9,1.25);

        \filldraw[blue!50!white, draw=black, thick] (2.15,-3) rectangle (3.65,-1.5);
        \draw (2.9,-2.25) node[rectangle,scale=0.8] {$B_1$};
        \filldraw[blue!50!white, draw=black, thick] (4.4,1.25) rectangle (5.9,2.75);
        \draw (5.15,2) node[rectangle,scale=0.8] {$B_2$};
        \filldraw[blue!50!white, draw=black, thick] (6.4,-1) rectangle (7.9,0.5);
        \draw (7.15,-0.25) node[rectangle,scale=0.8] {$B_3$};
        \draw (4.05,3.7) node[rectangle,scale=0.7] {$\alpha_1$};
        \draw (6.15,3.7) node[rectangle,scale=0.7] {$\alpha_2$};
        \draw (1.3,-1.25) node[rectangle,scale=0.7] {$\beta_2$};
        \draw (1.3,0.825) node[rectangle,scale=0.7] {$\beta_1$};

        \filldraw[yellow!50!white, draw=black, thick] (2,0.4) rectangle (2.8,1.9);
        \draw (2.4,1.15) node[rectangle,scale=0.8] {$A_1$};
        \filldraw[yellow!50!white, draw=black, thick] (3,0.3) rectangle (4.5,-0.4);
        \draw (3.75,-0.05) node[rectangle,scale=0.8] {$A_2$};
        \filldraw[yellow!50!white, draw=black, thick] (5.8,-3.3) rectangle (7.3,-2.5);
        \draw (6.55,-2.9) node[rectangle,scale=0.8] {$A_3$};
        
        \filldraw[brown!50!white, draw=black, thick] (4.6,-2.2) rectangle (5.1,-0.5);
        \draw (4.875,-1.35) node[rectangle,scale=0.6] {$C_1$};
        \filldraw[brown!50!white, draw=black, thick] (5.2,-2) rectangle (5.7,-0.7);
        \draw (5.45,-1.35) node[rectangle,scale=0.6] {$C_2$};
        \draw (8.6,-3.7) node[rectangle,scale=0.8] {(b)};
    \end{tikzpicture}
    \hspace{0.2cm}
    \begin{tikzpicture}[scale = 0.195mm]
        \filldraw[brown!10!white] (3.65,-4) rectangle (4.4,4);
        \filldraw[brown!10!white] (5.9,-4) rectangle (6.4,4);
        \filldraw[brown!10!white] (1,-1.5) rectangle (9,-1);
        \filldraw[brown!10!white] (1,0.5) rectangle (9,1.25);
        
        \draw[dashed, very thin] (3.65,4) -- (3.65,-4);
        \draw[dashed, very thin] (4.4,4) -- (4.4,-4);

        \draw[dashed, very thin] (5.9,4) -- (5.9,-4);
        \draw[dashed, very thin] (6.4,4) -- (6.4,-4);

        \draw[dashed, very thin] (1,-1.5) -- (9,-1.5);
        \draw[dashed, very thin] (1,-1) -- (9,-1);

        \draw[dashed, very thin] (1,0.5) -- (9,0.5);
        \draw[dashed, very thin] (1,1.25) -- (9,1.25);

        \filldraw[blue!50!white, draw=black, thick] (2.15,-3) rectangle (3.65,-1.5);
        \draw (2.9,-2.25) node[rectangle,scale=0.8] {$B_1$};
        \filldraw[blue!50!white, draw=black, thick] (4.4,1.25) rectangle (5.9,2.75);
        \draw (5.15,2) node[rectangle,scale=0.8] {$B_2$};
        \filldraw[blue!50!white, draw=black, thick] (6.4,-1) rectangle (7.9,0.5);
        \draw (7.15,-0.25) node[rectangle,scale=0.8] {$B_3$};
        \draw (4.05,3.7) node[rectangle,scale=0.7] {$\alpha_1$};
        \draw (6.15,3.7) node[rectangle,scale=0.7] {$\alpha_2$};
        \draw (1.3,-1.25) node[rectangle,scale=0.7] {$\beta_2$};
        \draw (1.3,0.825) node[rectangle,scale=0.7] {$\beta_1$};

        \filldraw[yellow!50!white, draw=black, thick] (4.9,-0.8) rectangle (5.7,-2.3);
        \draw (5.3,-1.55) node[rectangle,scale=0.8] {$A_2$};
        \filldraw[yellow!50!white, draw=black, thick] (3.55,0.2) rectangle (4.8,-0.6);
        \draw (4.18,-0.2) node[rectangle,scale=0.8] {$A_1$};
        \filldraw[yellow!50!white, draw=black, thick] (5.8,-3.3) rectangle (7.3,-2.5);
        \draw (6.55,-2.9) node[rectangle,scale=0.8] {$A_3$};
        
        \filldraw[brown!50!white, draw=black, thick] (1.6,1.9) rectangle (2.4,0.4);
        \draw (2,1.15) node[rectangle,scale=0.8] {$C_2$};
        \filldraw[brown!50!white, draw=black, thick] (2.6,3) rectangle (3.4,0.3);
        \draw (3.03,1.65) node[rectangle,scale=0.8] {$C_1$};
        \draw (8.6,-3.7) node[rectangle,scale=0.8] {(c)};
    \end{tikzpicture}
    \caption{Examples of three possible representations of $G'$}
    \label{fig:TRVG_of_G'}
\end{figure}
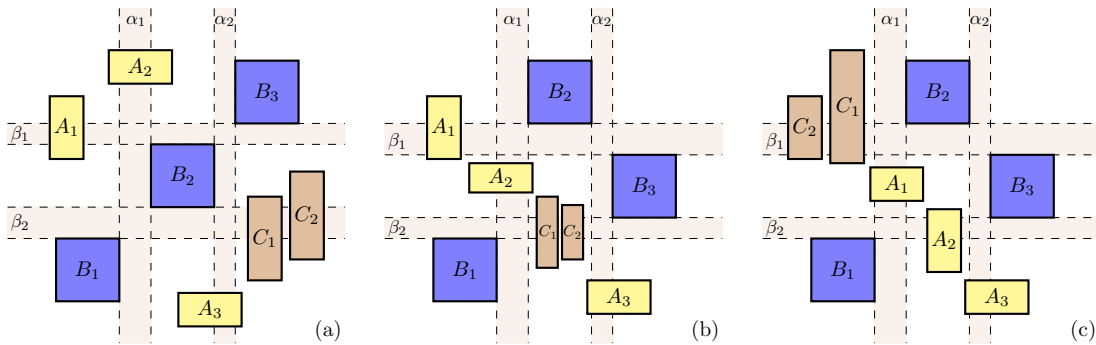

Next, we consider the placements of $D$-rectangles and $E$-rectangles. This is done using the concept of being entirely covered by another rectangle. A rectangle $A$ is \textit{entirely vertically covered} by a rectangle $B$ if any vertical line intersecting the interior of $A$ intersects the interior of $B$. If $A$ sees $B$ vertically, but the condition is not met, we will call that $A$ is \textit{partially vertically covered} by $B$. The entire and partial horizontal covering are defined similarly. With this definition, we present the following claim.

\begin{claim}\label{claim:ITRVG}
    Regarding the representations shown in Figure \ref{fig:TRVG_of_G'}, there is a $C$-rectangle that sees all $D$-rectangles horizontally. Moreover, at least two of these $D$-rectangles are entirely horizontally covered by this $C$-rectangle.
\end{claim}
\begin{proof}
    Since the rectangles representing $C_1$ and $C_2$ contain the same horizontal strip and this strip only, they see each other horizontally and see a $B$-rectangle, say $B'$, vertically. Note that any $B$-rectangle has at least one side adjacent to a vertical strip which is already contained in an $A$-rectangle. It means that the interior of any $C$-rectangle cannot intersect with this vertical strip; otherwise it would see an $A$-rectangle. Thus, there is at most one way for $C$-rectangles to be partially vertically covered by $B'$ \textemdash\ by having the interior extended beyond the side of $B$ that is not adjacent to a vertical strip. Hence, if both $C$-rectangles were partially vertically covered by $B'$, they must share a common vertical line of sight passing the right side of $B'$. Therefore, they would see each other vertically, which is impossible because they already see each other horizontally. It follows that one of them, say $C_1$, must be entirely vertically covered by $B'$. Now, observe that if a rectangle sees $C_1$ vertically, then it also sees $B'$. This fact forces all $D$-rectangles to see $C_1$ horizontally since each $D$-rectangle must see $C_1$ while it must not see any $B$-rectangle. This proves the first part of the claim.

    For the latter part, we shall consider $C_1$ which is entirely vertically covered by $B'$ in each case. In the first case of possible representation of $G'$ (Figure \ref{fig:TRVG_of_G'} (a)), we can see that $C_1$ could not have its top side higher than the top side $B_2$ as it would see $A_1$ which is prohibited. Likewise, the bottom side of $C_1$ could not be lower than the bottom side of $B_1$, otherwise it would see $A_3$. Thus, in order not to see $B$-rectangles, each $D$-rectangle must lie within the strip $\beta_2$ which is contained in $C_1$. Hence, all $D$-rectangles are entirely horizontally covered by~$C_1$.

    The second case (Figure \ref{fig:TRVG_of_G'} (b)) can be reasoned similarly to the first case. 

    Now consider the third case (Figure \ref{fig:TRVG_of_G'} (c)). As $C_1$ could have its top side higher than the top side of $B_2$, while its bottom side must not be lower than the bottom side of $B_3$, the possible area of $D$-rectangle placement are within $\beta_1$ or higher than the top side of $B_2$. Note that with these possible locations, a $D$-rectangle partially horizontally covered by $C_1$ must have a horizontal line of sight passing the top side of $C_1$. This allows only one $D$-rectangle to be partially horizontally covered by $C_1$ since they must not see each other. Hence, the rest of them, at least two $D$-rectangles, are entirely horizontally covered by $C_1$.
\end{proof}

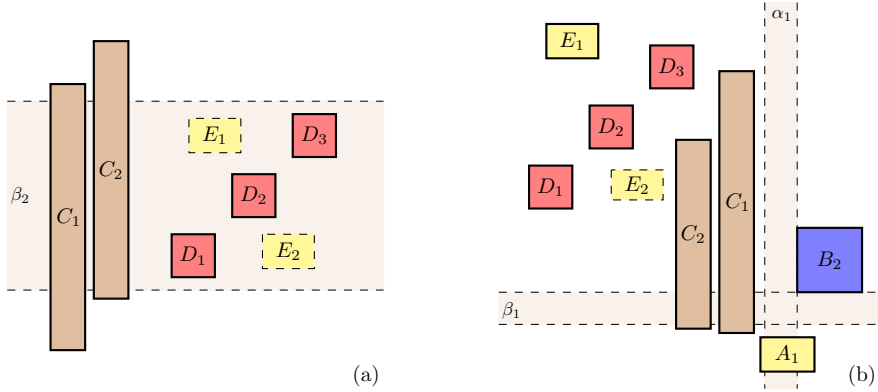
\begin{figure}[h]
    \centering
    \begin{tikzpicture}[scale = 0.2mm]
        \filldraw[brown!10!white] (-2.4,1.8) rectangle (6.3,6.2);

        \draw[dashed, very thin] (-2.4,1.8) -- (6.3,1.8);
        \draw[dashed, very thin] (-2.4,6.2) -- (6.3,6.2);

        \draw (-2.1,4) node[rectangle,scale=0.7] {$\beta_2$};

        \filldraw[red!50!white, draw=black, thick] (4.2,5.9) rectangle (5.2,4.9);
        \draw (4.7,5.4) node[rectangle,scale=0.8] {$D_3$};
        \filldraw[red!50!white, draw=black, thick] (2.8,4.5) rectangle (3.8,3.5);
        \draw (3.3,4) node[rectangle,scale=0.8] {$D_2$};
        \filldraw[red!50!white, draw=black, thick] (1.4,3.1) rectangle (2.4,2.1);
        \draw (1.9,2.6) node[rectangle,scale=0.8] {$D_1$};
        
        \filldraw[brown!50!white, draw=black, thick] (-0.4,7.6) rectangle (0.4,1.6);
        \draw (0,4.6) node[rectangle,scale=0.8] {$C_2$};
        \filldraw[brown!50!white, draw=black, thick] (-1.4,6.6) rectangle (-0.6,0.4);
        \draw (-0.975,3.5) node[rectangle,scale=0.8] {$C_1$};

        \filldraw[yellow!50!white, draw=black, dashed] (3.5,3.1) rectangle (4.7,2.3);
        \draw (4.1,2.7) node[rectangle,scale=0.8] {$E_2$};
        \filldraw[yellow!50!white, draw=black, dashed] (3,5.8) rectangle (1.8,5);
        \draw (2.4,5.4) node[rectangle,scale=0.8] {$E_1$};

        \draw (5.9,-0.2) node[rectangle,scale=0.8] {(a)};
    \end{tikzpicture}
    \hspace{1.2cm}
    \begin{tikzpicture}[scale = 0.2mm]
        \filldraw[brown!10!white] (3.65,-1) rectangle (4.4,8);
        \filldraw[brown!10!white] (-2.5,0.5) rectangle (6.3,1.25);
        
        \draw[dashed, very thin] (3.65,8) -- (3.65,-1);
        \draw[dashed, very thin] (4.4,8) -- (4.4,-1);

        \draw[dashed, very thin] (-2.5,0.5) -- (6.3,0.5);
        \draw[dashed, very thin] (-2.5,1.25) -- (6.3,1.25);

        \filldraw[blue!50!white, draw=black, thick] (4.4,1.25) rectangle (5.9,2.75);
        \draw (5.15,2) node[rectangle,scale=0.8] {$B_2$};
        \draw (4.05,7.7) node[rectangle,scale=0.7] {$\alpha_1$};
        \draw (-2.2,0.825) node[rectangle,scale=0.7] {$\beta_1$};

        \filldraw[red!50!white, draw=black, thick] (1,7) rectangle (2,6);
        \draw (1.5,6.5) node[rectangle,scale=0.8] {$D_3$};
        \filldraw[red!50!white, draw=black, thick] (-0.4,5.6) rectangle (0.6,4.6);
        \draw (0.1,5.1) node[rectangle,scale=0.8] {$D_2$};
        \filldraw[red!50!white, draw=black, thick] (-1.8,4.2) rectangle (-0.8,3.2);
        \draw (-1.3,3.7) node[rectangle,scale=0.8] {$D_1$};

        \filldraw[yellow!50!white, draw=black, thick] (3.55,0.2) rectangle (4.8,-0.6);
        \draw (4.18,-0.2) node[rectangle,scale=0.8] {$A_1$};
        
        \filldraw[brown!50!white, draw=black, thick] (1.6,4.8) rectangle (2.4,0.4);
        \draw (2,2.6) node[rectangle,scale=0.8] {$C_2$};
        \filldraw[brown!50!white, draw=black, thick] (2.6,6.4) rectangle (3.4,0.3);
        \draw (3.03,3.35) node[rectangle,scale=0.8] {$C_1$};

        \filldraw[yellow!50!white, draw=black, dashed] (0.1,4.1) rectangle (1.3,3.4);
        \draw (0.7,3.75) node[rectangle,scale=0.8] {$E_2$};
        \filldraw[yellow!50!white, draw=black, thick] (-0.2,7.5) rectangle (-1.4,6.7);
        \draw (-0.8,7.1) node[rectangle,scale=0.8] {$E_1$};

        \draw (5.9,-0.7) node[rectangle,scale=0.8] {(b)};
    \end{tikzpicture}
    \caption{Examples of placement of $D$-rectangles and $E$-rectangle that results in $C_1$ seeing some $E$-rectangles (shown as the dashed ones) horizontally: (a) for the cases of Figure \ref{fig:TRVG_of_G'} (a) and (b), and (b) for the case of Figure \ref{fig:TRVG_of_G'} (c)}
    \label{fig:D&E-rectangles}
\end{figure}

Next, we consider the strips formed by $D$-rectangles (whose placement are considered analogously to those of $B$-rectangles). As a $C$-rectangle see all $D$-rectangles horizontally, both horizontal strips are contained by that $C$-rectangle. This means that each $E$-rectangle, in order to see all $D$-rectangle and not see this $C$-rectangle, contains exactly one vertical strip, different from the other. Regarding the properties of the strips, each $E$-rectangle will see one $D$-rectangle horizontally, different from the other. Hence, we now have two $D$-rectangle that are seen by $E$-rectangles horizontally. Together with the latter part of Claim \ref{claim:ITRVG}, we can conclude that there is at least one $D$-rectangle that is seen by an $E$-rectangle while being entirely horizontally covered by a $C$-rectangle. This leads to this $E$-rectangle being seen by a $C$-rectangle (for example, see Figure \ref{fig:D&E-rectangles}) which is forbidden. Hence, it is impossible to construct a TRVG representation of $G$; therefore, $G$ is a non-TRVG.
\end{proof}

As each TRVG representation can be considered as an ITRVG representation of a graph, while there is an example of an ITRVG that is not a TRVG, we finally obtain the following theorem.
\begin{theorem}
    The class of TRVGs is a proper subclass of the class of ITRVGs.
    %Any TRVG is an ITRVG; however, he converse does not hold.
\end{theorem}

\section{Concluding Remarks}\label{conclude}

Through the generalization of rectangle visibility graphs (RVG) by allowing transparency and intersections, we have obtained new results. Regarding transparent rectangle visibility graphs (TRVG), it is proven that $K_{3,3,3}$ is not a TRVG. Through generalizing results from \cite{CJTK}, we have shown that a complete $n$-partite graph $K_{a_1,a_2, \dots, a_n}$ such that $a_1 \leq a_2 \leq \dots \leq a_n$ is a TRVG if and only if $a_{n-1} \leq 2$ or $(a_{n-2},a_{n-1},a_n) \in \{(1,3,3),(1,3,4),(2,3,3),(2,3,4)\}$. Next, we have shown that the complement of $C_n^2$, denoted as $D_n^2$, is not a TRVG whenever $n \geq 15$, and that $k$-partite TRVGs with $n$ vertices have at most $2(k-1)n-k(k-1)$ edges. Note that for the cases of $10 \leq n \leq 14$, it is still uncertain whether $D^2_n$ is a TRVG or not. Finally, we introduced intersecting transparent rectangle visiblity graphs (ITRVG), which have proven to be a strict generalization of TRVGs, as there exists a graph which is an ITRVG but not a TRVG. 

Here we leave some open problems regarding ITRVGs and TRVGs.

\begin{problem}
  Find a necessary condition for an ITRVG to be a TRVG.
\end{problem}

\begin{problem}
    Find other graphs that are ITRVGs but not TRVGs.
\end{problem}

Since these results are not trivially gained, it implies that the generalizing of RVGs into TRVGs or ITRVGs is a demanding process, and that there is more to be discovered in these areas. Some open problems regarding TRVGs are listed in \cite{CJTK}. It is also possible that the results for TRVGs may be useful for computer science-related fields, as RVGs have seen usage for VLSI systems.

\section*{Acknowledgement}
The research of this paper was carried out while the authors were participating in the Geometry and Combinatorics Boot Camp 2025, in Bangkok, Thailand. We are grateful for the hospitality of the Department of Mathematics and Computer Science, Faculty of Science, Chulalongkorn University.

\bibliographystyle{siam}
\bibliography{TRVGs}

\end{document}